\documentclass{article}
\usepackage{amsmath}
\usepackage{amsfonts}
\usepackage{amssymb}
\usepackage{amstext}
\usepackage{amsbsy}
\usepackage{amsopn}
\usepackage{amsthm}
\usepackage{amsxtra}
\usepackage{upref}
\usepackage{fancyhdr,amsmath, graphicx, psfrag, color ,amsfonts,layout, amssymb}
\usepackage{euscript,epsfig}

\newtheorem{thm}{Theorem}[section]

\newtheorem{prop}[thm]{Proposition}

\theoremstyle{definition}

\theoremstyle{remark}

\def \R{\mathbb R}
\def \N{{\mathbb N}}

\def \H{\mathbb H}

\begin{document}

\title{A short proof of unique ergodicity of horospherical foliations on infinite volume hyperbolic manifolds} 
\author{  Barbara SCHAPIRA}
 
\maketitle

\begin{abstract}
\end{abstract}

\section{Introduction}

The (unstable) horocycle flow on the unit tangent bundle of compact hyperbolic surfaces is uniquely ergodic. 
Furstenberg \cite{Furstenberg} proved that the Liouville measure is the unique invariant measure under this flow. 
This result has been extended to  many noncompact situations. 
On finite volume hyperbolic surfaces, Dani \cite{Dani78} proved that the Liouville measure 
is the unique finite invariant ergodic measure, except the measures supported on periodic orbits. 
On convex-cocompact hyperbolic surfaces, Burger \cite{Burger} proved that there is a unique locally finite ergodic invariant measure. 
It also follows from a result of Bowen-Marcus establishing the unique ergodicity of strong (un)stable
 Babillot-Ledrappier \cite{BL96} and Sarig \cite{Sarig} described completely the set of invariant ergodic measures
 of the horocyclic flow on abelian covers of compact hyperbolic surfaces. 

Roblin \cite{Roblin} proved in a much more general context that the unstable horocyclic flow 
admits a unique invariant ergodic Radon measure whose support is the set
of vectors whose negative geodesic orbit comes back infinitely often to a compact set, as soon as the 
geodesic flow admits a finite measure of maximal entropy. 

The goal of this note is to provide a new simpler proof of his result, 
inspired by the arguments of \cite{Coudene} in the case of surfaces of finite volume, 
with  additional ingredients to deal with the fact that when the manifold has infinite volume, 
there is in general no ergodic invariant measure which is invariant and ergodic under both
 the geodesic flow and its strong (un)stable foliation. 

Let us state the result with more precisions. 

Let $M=\Gamma\backslash\H^n$ be a hyperbolic manifold of dimension $n$ 
with infinite manifold. Let $G=SO^o(n,1)$ the group 
of isometries preserving orientation of $\H^n$, and $G=NAK$ its Iwasawa decomposition. 
The homogeneous space $\Gamma\backslash G$ is the frame bundle over $M=\Gamma\backslash G/K$. 
The action of $A $ by right multiplication on $\Gamma\backslash G$ is the natural lift of the action of
 the geodesic flow on the unit tangent bundle $T^1M$: 
$A $ moves the first vector $v_1$ of a frame $F=(v_1,\dots, v_n)$ as the geodesic flow    on $T^1M$ does, 
and the other vectors of the frame follow by parallel transport along the geodesic of $v_1$. 
The $N$-orbits on $\Gamma\backslash G$ are the strong unstable manifolds of this $A$ action, and project  to $T^1M$ 
onto the strong unstable leaves of the geodesic flow. 

We give a simple proof of the following result due to Roblin in a more general context (geodesic flows of $CAT(-1)$ spaces). 
We state it on $\Gamma\backslash G$ instead of $T^1M$. 

\begin{thm}[Roblin  \cite{Roblin}] Let $M=\Gamma\backslash\H^n$ be a hyperbolic manifold. 
Assume that $\Gamma$ is Zariski dense and that the geodesic flow on $T^1M$ 
admits a probability measure maximizing entropy. 
Then there is a unique (up to multiplicative constants) $N$-invariant conservative measure giving full measure to the
set 
$$
\Omega_{rad}^\mathcal{F}=\{F\in \Gamma\backslash G, a_{-t}F \mbox{ returns infinitely often in a compact set}\,\}
$$
\end{thm}

I thank Hee Oh for the reference \cite{Hochman}, which allows a proof valid in any dimension. 

\section{Infinite volume manifolds, actions of $A$ and $N$ }

Let $G=SO^o(n,1)$ be the group of direct isometries of the hyperbolic $n$-space $\H^n$. 
Let $M=\Gamma\backslash\H^n$ be a hyperbolic manifold, where $\Gamma$ is a discrete group 
without torsion. 

The {\em limit set} $\Lambda_\Gamma:=\overline{\Gamma.x}\setminus\Gamma.x$ is the set of accumulation
points of any orbit of $\Gamma$ in the boundary $\partial \H^n$. 
We assume $\Gamma$ to be non elementary, i.e. it is not virtually abelian, or equivalently, 
the set $\Lambda_\Gamma$ of its limit points in the boundary is infinite. 

Let $K=SO(n)$ be the stabilizer in $G$ of the point $o=(0,\dots, 0,1)\in \H^n$, 
and $L=SO(n-1)$ the stabilizer of the unit vector $(0,\dots, 0, 1)$ based at $o$. 
Then the unit tangent bundle $T^1M$ identifies with $\Gamma\backslash G/L$ whereas the homogeneous space
$\Gamma\backslash G$ identifies with the frame bundle over $T^1M$, whose fiber at any point is isomorphic to $L$. 
Denote by $\pi:T^1M\to M$ the natural projection.

The Busemann function is defined on $\partial\H^n\times\H^n\times\H^n$ by
$$
\beta_\xi(x,y)=\lim_{z\to\xi}d(x,z)-d(y,z)\,.
$$
The following map is an homeomorphism from $T^1\H^n$ to $\partial \H^n\times\partial\H^n\setminus\{diagonal\}\times \R$\,:
$$v\mapsto (v^-,v^+,\beta_{v^-}(\pi(v),o))\,.$$
In these coordinates, the geodesic flow acts by translation on the $\R$-coordinate, and an isometry $\gamma$ of $\H^n$ acts
as follows: $\gamma.(v^-,v^+,s)=(v^-,v^+,s\pm\beta_{v^-}(o,\gamma^\pm o))$. 
This homeomorphism induces on the quotient a homeomorphism from $T^1M$ to $\Gamma\backslash\left(\partial \H^n\times\partial\H^n\setminus\{diagonal\}\times \R\right)$.

The strong unstable manifold $W^{su}(v)=\{w\in T^1M, d(g^{-t}v, g^{-t}w)\to 0 \mbox{ when }t\to +\infty\}$ 
of a vector $v=(v^-,v^+,s)\in T^1\H^n$ under the geodesic flow is exactly the set of vectors $w=(v^-,w^+,s)$. 
The strong unstable manifold of a frame $F\in G$ under the action of $A$ is exactly its $N$-orbit.  \\

The nonwandering set of the geodesic flow $\Omega\in T^1M$ is exactly the set of vectors $v\in T^1M$ such that
$v^\pm \in\Lambda_\Gamma$. The nonwandering set $\Omega_\mathcal{F}$ of the action of $A$ in $\Gamma\backslash G$ is simply the set of frames 
whose first vector is in $\Omega$. 
The point is that $\Omega_\mathcal{F}$ is {\em not} $N$-invariant. 
Let $\mathcal{E}_\mathcal{F}=N.\Omega_\mathcal{F}$ and $\mathcal{E}$ its projection, that is the set of vectors $v\in T^1M$ such that $v^-\in\Lambda_\Gamma$.  

The difficulty in general is to deal with $\Omega_\mathcal{F}$ and $\mathcal{E}_\mathcal{F}$, to get informations on the dynamics of 
$N$ on $\mathcal{E}_\mathcal{F}$ thanks to the knowledge of the $A$ action on $\Omega_\mathcal{F}$ 'and vice versa).


\section{Ergodic theory}\label{ergodic}

The Patterson-Sullivan $\delta_\Gamma$-conformal  density is a family $(\nu_x)$ of equivalent measures 
on the boundary giving full measure to $\Lambda_\Gamma$, and satisfying the two crucial properties, 
for all $\gamma\in\Gamma$ and $\nu_x$-almost all $\xi\in\Lambda_\Gamma$ : 
$$
\gamma_*\nu_x=\nu_{\gamma x}\quad\mbox{and}\quad \frac{d\nu_y}{d\nu_x}(\xi)
=\exp(\delta_\Gamma\beta_\xi(x,y))
$$

The family $(\lambda_x)$ of Lebesgue measures on the unit spheres $T^1_x\H^n$, seen as measures on the boundary,
 satisfy a similar property:  for all $g\in G$ and $\lambda_x$-almost all $\xi\in\Lambda_\Gamma$ : 
$$
g_*\lambda_x=\lambda_{g x}\quad\mbox{and}\quad \frac{d\lambda_y}{d\lambda_x}(\xi)=\exp((n-1)\beta_\xi(x,y))
$$

We define a $\Gamma$-invariant measure $\tilde{m}_{BM}$
on $T^1\H^n=G/L$, in terms of the Hopf coordinates,  by 
$$
d\tilde{m}_{BM}(v)=e^{\pm \left(\delta_\Gamma\beta_{v^+}(o,\pi(v))+\delta_\Gamma\beta_{v^-}(o,\pi(v))\right)} d\nu_o(v^-)\,d\nu_o(v^+)dt
$$ 
This measure is also invariant under the geodesic flow. 
We denote $m_{BM}$ the induced measure on $T^1M$. 
This measure lifts in a natural way to the frame bundle, by taking locally its product with the Haar measure on $L$. 
We  denote it by $m^\mathcal{F}_{BM}$. 
 
The measure $m^\mathcal{F}_{BM}$ on $\Gamma\backslash G$ is not invariant under the $N$-action. 
However, its product structure allows to build such a measure. 
The Burger-Roblin measure is defined on $T^1\H^n$ by 

$$
d\tilde{m}_{BR}(v)=e^{\pm \left((n-1)\beta_{v^+}(o,\pi(v))+\delta_\Gamma\beta_{v^-}(o,\pi(v))\right)}d\nu_o(v^-)\,d\lambda_o(v^+)dt
$$
It is $\Gamma$-invariant, quasi-invariant under the geodesic flow, and we denote by $m_{BR}$ the 
induced measure on the quotient. 
Its lift $m^\mathcal{F}_{BR}$ to the frame bundle (by doing the local product with the Haar measure of $L$, as above) is 
$N$-invariant.

In any local chart $B$ of the strong unstable foliation of $A$, whose leaves are here the $N$-orbits, 
these measures have a very similar form, the product of the same transverse measure $\nu_T$, by a measure on the leaves. 
The transverse measure $(\nu_T)$ is a collection of measures on all transversals to the foliation by $N$-orbits, which is invariant by holonomy.  
In the chart $B$ of the foliation, denote by $T$ a transversal, and for $t\in T$, let $N(t)$ be the local leaf of the foliation intersected with $B$. 
For all continuous functions $\varphi:\Gamma\backslash G\to \R$ supported in $B$, we have
$$
\int \varphi\,dm_{BR}=\int_T \int_{N(t)} \varphi(F) d\lambda_{N.t}(F)d\nu_T(t)\,,$$
whereas
$$
\int \varphi\,dm_{BM}=\int_T \int_{N(t)} \varphi(F) d\mu^{BM}_{N.t}(F)d\nu_T(t)\,,
$$
where $d\lambda_{N.t}=dn$ and $d\mu_{N.t}^{BM}$ are respectively the conditional measures
of $m_{BR}$ and $m_{BM}$ on the $N$-orbits.

When it is finite, the measure $m_{BM}$ on $T^1M$ is mixing. 
Filling a gap in \cite{FlaminioSpatzier} Flaminio-Spatzier, D. Winter \cite{Winter} proved that its lift to $\Gamma\backslash G$ is also mixing, as soon as
the group $\Gamma$ is Zariski dense. 

As a corollary, he gets the following equidistribution result of averages pushed by the flow. 
Let $F\in\Omega_\mathcal{F}$ be a frame, and $\varphi:\Gamma\backslash G$ be a continuous map with compact support. 
Define  the following averages 
$$
M_1^t(\varphi)(R)=\frac{1}{\mu_{N.F}^{BM}(B_N(F,1))}\int_{B_N(F,1)}\varphi(a_t.X)\,d\mu_{N.F}^{BM}(X)\,.
$$
\begin{thm}[Winter \cite{Winter}]\label{mixing} Let $\Gamma<SO^o(n,1)$ be a Zariski dense discrete subgroup, such that 
the Bowen-Margulis measure $m_{BM}$ is finite. Then it is mixing.
As a consequence, for all $F\in\Omega_\mathcal{F}$ and $\varphi$ continuous with compact support, 
the averages $M_1^t(\varphi)(F)$ converge towards $\int_{\Gamma\backslash G}\varphi\,dm_{BM}$ when $t\to +\infty$. 
\end{thm}
 
We will need an ergodic theorem for the $N$-action on $\Gamma\backslash G$. As the natural $N$-invariant measure $m_{BR}^\mathcal{F}$ is infinite, 
we need a version of Hopf ratio ergodic theorem for actions of $\R^d$. The desired result is the following. 

\begin{thm}[Hochman  \cite{Hochman}]\label{Hopf} Consider a free, ergodic measure preserving action $(\phi_t)_{t\in\R^d}$ of $\R^d$ on a standard $\sigma$-finite measure space $(X,\mathcal{B},\nu)$. 
Let $\|.\|$ be any norm on $\R^d$ and $B_r=\{t\in\R^d, \|t\|\le r\}$. Then for every $f, g\in L^1(X,\nu)$, with $\int g\,d\nu\neq 0$ we have
$$
\frac{\int_{B_r} f(\phi^tx) dt}{\int_{B_r}g(\phi^tx) dt}\to \frac{\int_X f d\nu}{\int_X g d\nu }
$$
almost surely. 
\end{thm}

\section{Proof of the unique ergodicity} 

Assume that $\nu$ is a $N$ invariant ergodic and conservative measure on 
$\mathcal{E}_{rad}=\{F\in \mathcal{E}_\mathcal{F}\subset \Gamma\backslash G, a_{-t}F \mbox{ comes back i.o. in a compact set}\}$. 

Choose a generic frame $F\in \mathcal{E}_{rad}$ w.r.t $\nu$, i.e. a frame whose $N$ orbit becomes equidistributed towards $\nu$. 
Without loss of generality, translating $F$ along its $N$ orbit, 
we can assume that $F\in \Omega_{rad}=\{F\in \mathcal{E}_\mathcal{F}\subset \Gamma\backslash G, a_{-t}F \mbox{ comes back i.o. in a compact set}\}$. 

Therefore, we know that there exists a sequence $t_k\to +\infty$, such that $a_{-t_k}F$ converges to some frame $F_\infty\in\Omega_\mathcal{F}$. 

Without loss of generality, we can assume that
 $$
\mu_{N.F_\infty}(\partial B_N(F_\infty,1))=0\,,\quad\mbox{and for all } k\in\N \quad \mu_{N.a_{-t_k}F }(\partial B_N(a_{-t_k}F ,1))=0
$$ 
Indeed, if it were not the case, 
as $\mu_{N.F_\infty}$ and $\mu_{N.a_{-t_k}F}$ for all $k$ are Radon measure, there are at most countably many radii $r$ such that
$\mu_{NF_\infty}(\partial B_N(F_\infty,r))>0$ or $\mu_{N.a_{-t_k}F}(\partial B_N(a_{-t_k}F,r))>0$.  Choose $\rho$ close to $1$ such that 
all these measures of boundaries of balls of radius $\rho$ are zero,
 and change $t_k$ into $t_k+\log \rho$, $F_\infty$ into $g^{-\log \rho}F_\infty$.
 As the measures $\mu_{N.a_{-\log \rho} F_\infty}$ and $(a_{-\log\rho})_*\mu_{N.F_\infty}$ are proportional, 
we get $\mu_{N.a_{-\log\rho}.F_\infty}(\partial B_N(a_{-\log\rho}.F_\infty,1))=0$, and similarly $\mu_{N.a_{-\log\rho}.a_{-t_k}F }(\partial B_N(a_{-\log\rho}.a_{-t_k}F,1))=0$.

Choose any (nonnegative) continuous maps with compact support $\varphi$ and $\psi$, 
s.t. $\int\psi\,d\nu>0$. 
We will prove that 
$$
\frac{\int\varphi\,d\nu}{\int\psi\,d\nu}=
\frac{\int\varphi\,dm_{BR}}{\int\psi\,dm_{BR}}
$$
It will imply the theorem. 

The first important ingredient is an equicontinuity argument. 

Let $B_N(F,1)$ be the $N$-ball around $F$ inside $N.F$. 
For any continuous map $\varphi$, $F\in \Gamma\backslash G$ and $t\ge 0$, define
$$
M_1^t(\varphi)(F)=\frac{1}{\mu_{N.F}^{BM}(B_N(F,1))}\int_{B_N(F,1)}\varphi(a_t.X)d\mu_{N.F}^{BM}(X)\,.
$$
\begin{prop}[Equicontinuity] Let $\varphi$ be any uniformly continuous function. 
For all $F\in\Gamma\backslash G$ such that $\mu_{N.F}(\partial B_{NF}(F,1))=0$,  the maps $F\mapsto M_1^t(\varphi)(F)$ are equicontinuous in $t\ge 0$. 
\end{prop}

\begin{proof} The result is relatively classical for surfaces at least. It is written in details on $T^1M$ 
  here \cite{schapira2004}. The assumptions in this reference are  slightly stronger, but the compactness assumption of \cite{schapira2004}
was used only to ensure uniform continuity of $\varphi$.

The extension to $\Gamma\backslash G$ does not change anything. Indeed, the fibers of the fiber bundle 
$\Gamma\backslash G\to T^1M$ are included in the (weak) stable leaves of the $A$ action. Therefore, the
argument still applies. 
We refer to \cite{schapira2004}, but the idea is as follows. If $F$ and $F'$ are very close along a weak stable leaf, the sets 
$a_t(B_N(F,1))$ and $a_t(B_N(F',1))$ remain at distance roughly $d(F,F')$ when $t\ge 0$. 

If they are close and belong to the same $N$-orbit, then the assumption on the boundary of the balls allows to ensure that their averages stay close for all $t\ge 0$.

If $F$ and $F'$ are close, in general, there exists a frame $F'' \in N.F$ on the stable leaf of $F'$, 
so that combining both arguments above allows to conclude to equicontinuity. 
\end{proof}

The next ingredient is mixing. As said in section \ref{ergodic} above, when $t\to +\infty$, for all $X\in \Omega_{rad}$, we have 
$M_1^t(\varphi)(X)\to \int_{T^1M}\varphi\,dm_{BM}$ and $M_1^t(\psi)(X)\to\int_{T^1M}\psi\,dm_{BM}$, 
uniformly on compact sets. \\

From the above equicontinuity argument, we deduce relative compactness.   
In particular, consider the compact
$K=\{F_\infty\}\cup\{a_{-t_k}F,k\in\N\}$. Each frame of $K$ satisfies the assumption on the measure of the boundary of the $N$-ball of radius $1$. 
Therefore, there exists a subsequence of $t_k$, still denoted by $t_k$, such that 
$
M_1^{t_k}(\varphi)(F')$ converges uniformly to $\int \varphi dm_{BM}$ on $K$. \\

Let us now do the observation that 
$$
M_{e^{t}}^0(\phi)(F)=M_{1}^t(\phi)(a^{-t}F)\quad\mbox{and} \quad M_{e^t}^0(\psi)(F)=M_1^t(\phi)(a^{-t}F)\,.$$
Therefore, as $a^{-t_k}F$ converges to $F_\infty$, and by the above uniform convergence on  the compact
$K=\{F_\infty\}\cup\{a_{-t_k}F,k\in\N\}$, we deduce that 
\begin{eqnarray}\label{convergence-BM}
\frac{\int_{B_{N}(F,e^{t_k})}\varphi d\mu_{N.F}^{BM}}{\int_{B_{N}(F,e^{t_k})} \psi\,d\mu_{N.F}^{BM}}
\to \frac{\int\varphi\,dm_{BM}}{\int\psi\,dm_{BM}}\,.
\end{eqnarray}

Now, consider a small chart of the foliation $B$, with boundary of measure zero, 
and $\varphi$, $\psi$ continuous maps supported in $B$. 
Then, the integral 
$M_1^{t_k}(\varphi)(a_{-t_k}F)$ can be rewritten as 
$$
\int_T\int_{N(t)}\varphi\,d\mu_{N.t}\,d\nu_{T,t_k}+R(t_k,\varphi)\,,
$$
where
$$
\nu_{T,t_k}=\frac{1}{\mu_{N.F}(B_{N}(F,e^{t_k})}\sum_{t\in T\cap B_N(F,e^{t_k})}\delta_t
$$
As we chose $F_\infty$ such that $\mu_{N.F_\infty}(\partial B(F_\infty, 1))=0$, we deduce, as in \cite{MaucourantSchapira}, that the error term $R(t_k,\varphi)$, which is bounded from above by 
$\|\varphi\|_\infty.\frac{\mu_{N.F}(B_N(F,e^{t_k}+r_0)\setminus B_N(F,e^{t_k}-r_0))}{\mu_{N.F}(B_N(F,e^{t_k}))}$, goes to $0$ when $t_k\to +\infty$. 

As $M_1^{t_k}$ converges to $m_{BR}$ it implies that for all transversals $T$ to the foliation  of $\Gamma\backslash G$ in $N$ orbits, 
 $\nu_{T,t_k}$ converges weakly to $\nu_T$. \\

Coming back to our assumptions, thanks to theorem \ref{Hopf} we know that for all continuous maps $\varphi, \psi\in C_c(\Gamma\backslash G)$ with $\int\psi\,d\nu>0$, 
we have 
$$
\frac{\int_{B_{N.F}(F,r)}\varphi(nF)dn}{\int_{B_{N.F}(F,r)}\psi(nF)dn}\to \frac{\int_{\Gamma\backslash G} \varphi d\nu}{\int_{\Gamma\backslash G} \psi\,d\nu}
$$
By a standard approximation argument, this convergence also holds for $\psi={\bf 1}_\mathcal{B}$ where $\mathcal{B}$ 
is a relatively compact chart of the foliation with $\nu(\partial \mathcal{B})=0$. \\

Consider now such a box $\mathcal{B}$, a transversal $T$ to the foliation into $N$-orbits inside $\mathcal{B}$, and maps $\varphi$ with support in $\mathcal{B}$. 
Decompose the above averages as 
$$
\frac{\int_{B_{N.F}(F,e^{t_k})}\varphi(nF)dn}{\int_{B_{N.F}(F,e^{t_k})}{\bf 1}_\mathcal{B}(nF)dn}=  
\int_T\int_{N(t)}\varphi(n.t)dn \,d\nu_{T,t_k,B}+R(t_k,\varphi,B)\,,
$$
where 
$\nu_{T,t_k,B}=\frac{\mu_{N.F}(B_{N}(F,e^{t_k})}{\int_{B_{N.F}(F,e^{t_k})}{\bf 1}_\mathcal{B}(nF)dn} \nu_{T,t_k}$ and $R(t_k,\varphi,B)$ is bounded from above by 
$$\|\varphi\|_\infty.\frac{\int_{B_{N.F}(F,e^{t_k}+r_0)\setminus B_{NF}(F,e^{t_k}-r_0)}{\bf 1}_\mathcal{B}(nF)dn}{\int_{B_{N.F}(F,e^{t_k})}{\bf 1}_\mathcal{B}(nF)dn}$$

As $\int_{B_{N.F}(F,e^{t_k})}{\bf 1}_\mathcal{B}(nF)dn$ is bounded by the Lebesgue measure of $B_{NF}(F,e^{t_k})$ which grows polynomially in $r=e^{t_k}$, 
we know that for some subsequence of $t_k$, still denoted by $t_k$, $R(t_k, \varphi, B)$ converges to $0$. 
For such a subsequence, we deduce the convergence of the measure $\nu_{T,t_k,B}$ towards some positive finite measure on $T$. 
By definition of $\nu_{T,t_k}$ and $\nu_{T,t_k,B}$ this transverse measure has to be proportional to the transverse measure $\nu_T$ induced by the Bowen-Margulis measure $m_{BM}$, so that $\nu$ is necessarily proportional to $m_{BR}$. 
 

\addcontentsline{toc}{section}{References}
\bibliography{biblio}

\providecommand{\bysame}{\leavevmode\hbox to3em{\hrulefill}\thinspace}
\providecommand{\MR}{\relax\ifhmode\unskip\space\fi MR }
\providecommand{\MRhref}[2]{%
  \href{http://www.ams.org/mathscinet-getitem?mr=#1}{#2}
}
\providecommand{\href}[2]{#2}
\begin{thebibliography}{Win14}

\bibitem[BL98]{BL96}
Martine Babillot and Fran{\c{c}}ois Ledrappier, \emph{Geodesic paths and
  horocycle flow on abelian covers}, Lie groups and ergodic theory ({M}umbai,
  1996), Tata Inst. Fund. Res. Stud. Math., vol.~14, Tata Inst. Fund. Res.,
  Bombay, 1998, pp.~1--32. \MR{1699356 (2000e:37029)}

\bibitem[Bur90]{Burger}
Marc Burger, \emph{Horocycle flow on geometrically finite surfaces}, Duke Math.
  J. \textbf{61} (1990), no.~3, 779--803. \MR{1084459 (91k:58102)}

\bibitem[Cou09]{Coudene}
Yves Coudene, \emph{A short proof of the unique ergodicity of horocyclic
  flows}, Ergodic theory, Contemp. Math., vol. 485, Amer. Math. Soc.,
  Providence, RI, 2009, pp.~85--89. \MR{2553211 (2010g:37051)}

\bibitem[Dan78]{Dani78}
S.~G. Dani, \emph{Invariant measures of horospherical flows on noncompact
  homogeneous spaces}, Invent. Math. \textbf{47} (1978), no.~2, 101--138.
  \MR{0578655 (58 \#28260)}

\bibitem[FS90]{FlaminioSpatzier}
L.~Flaminio and R.~J. Spatzier, \emph{Geometrically finite groups,
  {P}atterson-{S}ullivan measures and {R}atner's rigidity theorem}, Invent.
  Math. \textbf{99} (1990), no.~3, 601--626. \MR{1032882 (91d:58201)}

\bibitem[Fur73]{Furstenberg}
Harry Furstenberg, \emph{The unique ergodicity of the horocycle flow}, Recent
  advances in topological dynamics ({P}roc. {C}onf., {Y}ale {U}niv., {N}ew
  {H}aven, {C}onn., 1972; in honor of {G}ustav {A}rnold {H}edlund), Springer,
  Berlin, 1973, pp.~95--115. Lecture Notes in Math., Vol. 318. \MR{0393339 (52
  \#14149)}

\bibitem[Hoc10]{Hochman}
Michael Hochman, \emph{A ratio ergodic theorem for multiparameter non-singular
  actions}, J. Eur. Math. Soc. (JEMS) \textbf{12} (2010), no.~2, 365--383.
  \MR{2608944 (2011g:37006)}

\bibitem[MS14]{MaucourantSchapira}
Fran{\c{c}}ois Maucourant and Barbara Schapira, \emph{Distribution of orbits in
  {$\Bbb R^2$} of a finitely generated group of {${\rm SL}(2,\Bbb R)$}}, Amer.
  J. Math. \textbf{136} (2014), no.~6, 1497--1542. \MR{3282979}

\bibitem[Rob03]{Roblin}
Thomas Roblin, \emph{Ergodicit\'e et \'equidistribution en courbure
  n\'egative}, M\'em. Soc. Math. Fr. (N.S.) (2003), no.~95, vi+96. \MR{2057305
  (2005d:37060)}

\bibitem[Sar04]{Sarig}
Omri Sarig, \emph{Invariant {R}adon measures for horocycle flows on abelian
  covers}, Invent. Math. \textbf{157} (2004), no.~3, 519--551. \MR{2092768
  (2005k:37059)}

\bibitem[Sch04]{schapira2004}
Barbara Schapira, \emph{On quasi-invariant transverse measures for the
  horospherical foliation of a negatively curved manifold}, Ergodic Theory
  Dynam. Systems \textbf{24} (2004), no.~1, 227--255. \MR{2041270
  (2005d:37061)}

\bibitem[Win14]{Winter}
Dale Winter, \emph{Mixing of frame flow for rank one locally symmetric spaces
  and measure classification}, preprint, 2014.

\end{thebibliography}
\bibliographystyle{amsalpha}


\end{document}